\documentclass[11pt]{article}
\usepackage{}
\usepackage{mathrsfs}
\usepackage{amsfonts}
\usepackage{amsfonts,amssymb,mathrsfs}
\usepackage{amsmath,amscd}
\usepackage[dvips]{graphicx}
\usepackage[all]{xy}
\usepackage{graphicx}
\usepackage{fancyhdr}
\usepackage{mathptm,pslatex}
\usepackage{amsthm}

\begin{document}

\sloppy
\newtheorem{Def}{Definition}[section]
\newtheorem{Prop}[Def]{Proposition}
\newtheorem{Theo}[Def]{Theorem}
\newtheorem{Lem}[Def]{Lemma}
\newtheorem{Koro}[Def]{Corollary}
\theoremstyle{definition}
\newtheorem{Rem}[Def]{Remark}
\newtheorem{Exam}[Def]{Example}

\newcommand{\add}{{\rm add}}
\newcommand{\gd}{{\rm gl.dim }}
\newcommand{\dm}{{\rm dom.dim }}
\newcommand{\E}{{\rm E}}
\newcommand{\G}{{\rm G}}
\newcommand{\Mor}{{\rm Morph}}
\newcommand{\End}{{\rm End}}
\newcommand{\ind}{{\rm ind}}
\newcommand{\rsd}{{\rm res.dim}}
\newcommand{\rd} {{\rm rep.dim}}
\newcommand{\ol}{\overline}
\newcommand{\rad}{{\rm rad}}
\newcommand{\soc}{{\rm soc}}
\renewcommand{\top}{{\rm top}}
\newcommand{\pd}{{\rm proj.dim}}
\newcommand{\id}{{\rm inj.dim}}
\newcommand{\Fac}{{\rm Fac}}
\newcommand{\fd} {{\rm fin.dim }}
\newcommand{\DTr}{{\rm DTr}}
\newcommand{\cpx}[1]{#1^{\bullet}}
\newcommand{\D}[1]{{\mathscr D}(#1)}
\newcommand{\Dz}[1]{{\mathscr D}^+(#1)}
\newcommand{\Df}[1]{{\mathscr D}^-(#1)}
\newcommand{\Db}[1]{{\mathscr D}^b(#1)}
\newcommand{\C}[1]{{\mathscr C}(#1)}
\newcommand{\Cz}[1]{{\mathscr C}^+(#1)}
\newcommand{\Cf}[1]{{\mathscr C}^-(#1)}
\newcommand{\Cb}[1]{{\mathscr C}^b(#1)}
\newcommand{\K}[1]{{\mathscr K}(#1)}
\newcommand{\Kz}[1]{{\mathscr K}^+(#1)}
\newcommand{\Kf}[1]{{\mathscr  K}^-(#1)}
\newcommand{\Kb}[1]{{\mathscr K}^b(#1)}
\newcommand{\modcat}{\ensuremath{\mbox{{\rm -mod}}}}
\newcommand{\stmodcat}[1]{#1\mbox{{\rm -{\underline{mod}}}}}
\newcommand{\gr}[1]{#1\mbox{{\rm -{grmod}}}}
\newcommand{\pmodcat}[1]{#1\mbox{{\rm -proj}}}
\newcommand{\imodcat}[1]{#1\mbox{{\rm -inj}}}
\newcommand{\opp}{^{\rm op}}
\newcommand{\otimesL}{\otimes^{\rm\bf L}}
\newcommand{\rHom}{{\rm\bf R}{\rm Hom}}
\newcommand{\projdim}{\pd}
\newcommand{\Hom}{{\rm Hom}}
\newcommand{\Coker}{{\rm coker}\,\,}
\newcommand{ \Ker  }{{\rm Ker}\,\,}
\newcommand{ \Img  }{{\rm Im}\,\,}
\newcommand{\Ext}{{\rm Ext}}
\newcommand{\cov}{{\rm cov}}
\newcommand{\StHom}{{\rm \underline{Hom} \, }}

\newcommand{\gm}{{\rm _{\Gamma_M}}}
\newcommand{\gmr}{{\rm _{\Gamma_M^R}}}

\newcommand{\scr}[1]{\mathscr #1}
\newcommand{\al}[1]{\mathcal #1}

\def\vez{\varepsilon}\def\bz{\bigoplus}  \def\sz {\oplus}
\def\epa{\xrightarrow} \def\inja{\hookrightarrow}

\newcommand{\lra}{\longrightarrow}
\newcommand{\lraf}[1]{\stackrel{#1}{\lra}}
\newcommand{\ra}{\rightarrow}
\newcommand{\dk}{{\rm dim_{_{k}}}}

{\Large \bf
\begin{center}
Stable equivalences of Morita type for $\Phi$-Beilinson-Green algebras
\end{center}}
\medskip

\centerline{{\bf Shengyong Pan}}

\begin{center} Department of Mathematics,\\
 Beijing Jiaotong University, \\
100044 Beijing, People's Republic of  China \\
 E-mail:  shypan@bjtu.edu.cn \\
\end{center}
\bigskip

\renewcommand{\thefootnote}{\alph{footnote}}
\setcounter{footnote}{-1} \footnote{ $^*$ Corresponding author.
Email: shypan@bjtu.edu.cn.}
\renewcommand{\thefootnote}{\alph{footnote}}
\setcounter{footnote}{-1} 
\renewcommand{\thefootnote}{\alph{footnote}}
\setcounter{footnote}{-1} 

\date{}

\begin{abstract}
We present a method to construct new stable equivalences of
Morita type. Suppose that  a
$B$-$A$-bimodule $N$ define a stable equivalence of Morita type
between finite dimensional algebras $A$ and $B$. Then, for any generator $X$ of the $A$-module category and any finite admissible set $\Phi$ of natural numbers, the
$\Phi$-Beilinson-Green algebras $\scr G^{\Phi}_A(X)$ and
$\scr G^{\Phi}_B(N\otimes_AX)$ are stably equivalent of Morita type. In
particular, if $\Phi=\{0\}$, we get a known result in literature.
 As another consequence, we
construct an infinite family of derived equivalent algebras of the
same dimension and of the same dominant dimension such that they are
pairwise not stably equivalent of Morita type. Finally, we develop some techniques for proving that, if there is a graded stable equivalence of Morita type between graded algebras, then we can 
get a stable equivalence of Morita type between Beilinson-Green algebras associated with graded algebras.
\end{abstract}


\section{Introduction}

In the representation theory of algebras and groups, there are three fundamental equivalences: Morita, derived and stable equivalences. A cycle of interconnected questions and conjectures known as ``Morita invariances" have inspired a lot of work over the last few decades.
As is known, for the first two, there is a corresponding
Morita theory for each (see \cite{morita, Rickard}), while much less is known for the last. Recently, a
special class of stable equivalences, called stable equivalences of
Morita type, are introduced by Brou\'e \cite{B} in modular representations of
finite groups. They are close to derived equivalences, are induced by bimodules, and are shown to be of great interest in modern
representation theory since they preserve many homological and
structural invariants of algebras and modules (see, for example,
\cite{B, DM, krause, L, P, C3, C2}).  
Although we have a better understanding of stable equivalences of Morita type than on general stable equivalences, 
we still cannot give an answer for Auslander-Reiten conjecture (see, for Auslander-Reiten conjecture, \cite{ARS, LZZ}). 
In order to understand this
kind of equivalences, one has to know, examples and
basic properties of stable equivalences of Morita type as many as
possible. 

\smallskip

Only a few  methods so far using one-point extensions, endomorphism algebras and  Auslander-Yoneda algebras were known in
\cite{LX1, LX2, LX3, CPX}. Rickard's result that the
existence of derived equivalences for self-injective algebras imply
the one of stable equivalences of Morita type provides another way
to construct stable equivalences of Morita type. This method is no longer true for general finite-dimensional algebras
(see \cite{hx3} for some new advances in this direction). Moreover, unlike tensor products and trivial extensions
preserved by derived equivalences \cite{R3}, stable equivalences of Morita type do not preserve tensor products and trivial extensions \cite{LZZ}.
So, a
systematical method for constructing stable equivalences of Morita
type seems not yet to be available.

In this paper, we shall look for a more general method to construct new stable equivalences of
Morita type for general finite-dimensional algebras. The 
procedure has two flexibilities, one is the choice of generators,
and the other is the one of finite admissible sets. Thus this
construction provides a large variety of stable equivalences of
Morita type.

To state our first main result, let us recall the definition of
$\Phi$-Beilinson-Green algebras. Let $A$ be a
finite-dimensional algebra and $X$ an $A$-module. Then, for an
admissible set $\Phi$ of natural numbers, there is defined an
algebra $\scr G^{\Phi}_A(X)$, called the $\Phi$-Beilinson-Green algebra
of $X$ in \cite{Pan, PZ}, which is equal to $\bigoplus_{i,j\in
\Phi}\mbox{Ext}^{j-i}_A(X,X)$ as a vector space, and its multiplication
is defined in a natural way (see Subsection \ref{adm} for details).
Our main result reads as follows:

\begin{Theo}\label{thm1}
Let $A$ and $B$ be finite-dimensional $k$-algebras over a field $k$.
Assume that two bimodules $_{A}M _{B}$ and $_{B}N_{A}$ define a
stable equivalence of Morita type between $A$ and $B$. Let $X$ be an
$A$-module which is a generator for $A$-module category. Then, for
any finite admissible set $\Phi$ of natural numbers, there is a
stable equivalence of Morita type between $\scr G^{\Phi}_{A}(X)$ and
$\scr G^{\Phi}_{B}(N\otimes_{A}X)$. 
\end{Theo}

\begin{Rem} $(1)$ Note that if $\Phi=\{0\}$, then the above result was known in
\cite{CPX,LX3}. Thus Theorem \ref{thm1} generalizes the main result in
\cite{LX3}, and provides much more possibilities through $\Phi$ for
constructing stable equivalences of Morita type. Also, our proof of
Theorem \ref{thm1} is different from that in \cite{CPX,LX3}.

$(2)$ Note that  $\Phi$-Beilinson-Green algebra is a C-construction 
which is different from $\Phi$-Auslander-Yoneda algebra, however, $\Phi$-Auslander-Yoneda algebra is B-construction in the sense of Mori \cite{Mori}.
\end{Rem}

We have the
following characterization of stable equivalences of Morita type for self-injective
algebras from Theorem \ref{thm1} and \cite[Theorem 2]{CPX}.

\begin {Koro}\label{koro1} Suppose that $A$ and $B$ are finite-dimensional
self-injective $k$-algebras over a field $k$ such that neither $A$ nor $B$ has
semisimple direct summands. Let $X$ be an
$A$-module and let $Y$ be a $B$-module. If there is a finite
admissible set $\Phi$ of natural numbers such that
$\scr G^{\Phi}_A(A\oplus X)$ and $\scr G^{\Phi}_B(B\oplus Y)$ are stably
equivalent of Morita type, then, for \emph{any} finite admissible
set $\Psi$ of natural numbers, the algebras $\scr G^{\Psi}_A(A\oplus X)$
and $\scr G^{\Psi}_B(B\oplus Y)$ are stably equivalent of Morita type.
\end{Koro}

As another byproduct of our considerations in this paper, we can
construct a family of derived equivalent algebras with certain
special properties, but derived equivalent algebras are not stably equivalent of Morita type.

\begin{Koro} There is an infinite series of $k$-algebras of the same
dimension over an algebraically closed field such that they have the
same dominant and global dimensions, and are all derived equivalent,
but not pairwise stably equivalent of Morita type. \label{koro2}
\end{Koro}

Finally, we develop some techniques for proving that, for finite group $G$, we get stable equivalences of Morita type between Beilinson-Green algebra of $G$-graded algebra from $G$-graded stable equivalences of Morita type. For a $G$-graded $k$-algebra $A=\oplus_{g\in G}A_g$, there is a Beilinson-Green algebra $\bar{A}$ defined as a $G\times G$-matrix algebra with $(\bar{A}_{gh})_{g,h\in G}$, where 
$\bar{A}_{gh}=A_{gh^{-1}}$.

\begin{Theo} \label{5} Suppose that there is a $G$-graded stable equivalence of Morita type between $G$-graded algebras $A$ and $B$, then there is a stable equivalence of Morita type between the Beilinson-Green algebras $\bar{A}$ and $\bar{B}$ of $A$ and $B$, respectively.
\end{Theo}

We mention that the proof of Theorem \ref{5} is somewhat routine. However, the assertions seem to be quite surprising.

The  paper is structured as follows. In Section
\ref{pre}, we shall fix notations and prepare some basic facts for
our proofs. In Section \ref{induction}, we shall prove our main results, Theorem
\ref{thm1}. In Section \ref{selfinjective}, we
will concentrate our consideration on self-injective algebras, and
establish some applications of our main results. In particular, in
this section we shall prove Corollary \ref{koro1} and supply a
sufficient condition to verify when two algebras are not stably
equivalent of Morita type to another algebra, which will be used in
Section \ref{family}. In Section \ref{family}, we shall apply our
results in the previous sections to Liu-Schulz algebras and give a
proof of Corollary \ref{koro2} which answers a question by Thorsten
Holm. In Section \ref{BG}, for finite group, we can construct stable equivalences of Morita type between Beilinson-Green algebra of 
group graded algebra from group graded stable equivalences of Morita type. We will give an example to illustrate our main result in Section \ref{Example}.

{\bf Acknowledgements.} Shengyong Pan is funded by China Scholarship Council. He thanks Professors Hongxing Chen and Wei Hu for useful discussions and suggestions, and also thanks his wife Ran Wen, his sons Ruipeng and Ruiqi for their encouragements. This paper is done during a visit of Shengyong Pan to the University of Edinburgh, he would like to thank Professor Susan J. Sierra for her hospitality and useful discussions.

\section{Preliminaries\label{pre}}
In this section, we fix some notations, and recall some
definitions and basic results which are needed in the proofs of our
main results.

\subsection{Some conventions and homological facts}
Throughout this paper, $k$ stands for a fixed field. All categories
and functors will be $k$-categories and $k$-functors, respectively.
Unless stated otherwise, all algebras considered are
finite-dimensional $k$-algebras, and all modules are finitely
generated left modules.

Let $\mathcal C$ be a category. Given two morphisms $f: X\to Y$ and
$g: Y\to Z$ in $\mathcal C$, we denote the composition of $f$ and
$g$ by $fg$ which is a morphism from $X$ to $Z$, while we denote the
composition of a functor $F:\mathcal {C}\to \mathcal{D}$ between
categories $\mathcal C$ and $\mathcal D$ with a functor $G:
\mathcal{D}\ra \mathcal{E}$ between categories $\mathcal D$ and
$\mathcal E$ by $GF$ which is a functor from $\mathcal C$ to
$\mathcal E$.

If $\cal C$ is an additive category and $X$ is an object in
$\mathcal{C}$, we denote by $\add(X)$ the full subcategory of
$\mathcal{C}$ consisting of all direct summands of finite direct
sums of copies of $X$. The object $X$ is called an additive
generator for $\mathcal C$ if add($X$) = $\mathcal C$.

Let $A$ be an algebra. We denote by $A\modcat$ the category of all
$A$-modules, by $A\pmodcat$ (respectively, $A\imodcat$) the full
subcategory of $A\modcat$ consisting of projective (respectively,
injective) modules,  by $D$ the usual $k$-duality $\Hom_{k}(-, k)$,
and by $\nu_{A}$ the Nakayama functor $D\Hom_{A}(-,\,_{A}A)$ of $A$.
Note that $\nu_A$ is an equivalence from $A\pmodcat$ to $A\imodcat$
with the inverse $\Hom_A(D(A),-)$. We denote the global and dominant
dimensions of $A$ by gl.dim$(A)$ and dom.dim$(A)$, respectively.

As usual, by $\Db{A}$ we denote the bounded derived category of
complexes over $A$-mod. It is known that $A$-mod is fully embedded
in $\Db{A}$ and that $\Hom_{\Db{A}}(X,Y[i])\simeq \Ext^i_A(X,Y)$ for
all $i\ge 0$ and $A$-modules $X$ and $Y$.
Let $X$ be an $A$-module. We denote by $\Omega^i_A(X)$ the $i$-th
syzygy, by $\soc(X)$ the socle, and by rad$(X)$ the Jacobson radical
of $X$.

Let $X$ be an additive generator for $A$-mod. The endomorphism
algebra of $X$ is called the Auslander algebra of $A$. This algebra
is, up to Morita equivalence, uniquely determined by $A$. Note that
Auslander algebras can be described by two homological properties:
An algebra $A$ is an Auslander algebra if gl.dim($A$)$\leq 2\le
\dm(A)$.

An $A$-module $X$ is called a generator for $A$-mod if
$\add(_AA)\subseteq \add(X)$; a cogenerator for $A$-mod if
$\add(D(A_A))\subseteq \add(X)$, and a generator-cogenerator if it
is both a generator and a cogenerator for $A$-mod. Clearly, an
additive generator for $A$-mod is a generator-cogenerator for
$A$-mod. But the converse is not true in general.

\subsection{Admissible sets and perforated covering categories\label{adm}}

In \cite{Pan, PZ}, a class of algebras, called $\Phi$-Green
algebras, were introduced.
Recall that a subset $\Phi$ of $\mathbb N$ is said to be admissible
provided that $0\in \Phi$ and that for any $p,q,r\in \Phi$ with
$p+q+r\in \Phi$ we have $p+q\in \Phi$ if and only if $q+r\in \Phi$.
As shown in \cite{HAU}, there are a lot of admissible subsets of
$\mathbb N$. For example, given any subset $S$ of $\mathbb N$
containing $0$, the set $\{x^m\mid x\in S\}$ is admissible for all
$m\ge 3.$

Assume that  $\Phi$ is an admissible subset of $\mathbb N$.
Let $\mathcal C$ be a $k$-category, and let $F$ be an additive
functor from $\mathcal C$ to itself. The $(F,\Phi)$-covering category
${\mathcal C}_{\cov}^{F,\Phi}$ of $\mathcal C$ is a category in which the
objects are the same as that of $\mathcal C$, and the morphism set
between two objects $X$ and $Y$ is defined to be

$$ \Hom_{{\mathcal C}_{\cov}^{F,\Phi}}(X,Y):=\displaystyle \bigoplus_{i,j\in \Phi} \Hom_{\mathcal C}(F^iX, F^jY),$$
and the composition is defined in an obvious way. Since $\Phi$ is
admissible, ${\mathcal C}_{\cov}^{F,\Phi}$ is an additive $k$-category. In
particular, $\Hom_{{\mathcal C}_{\cov}^{F,\Phi}}(X, X)$ is a $k$-algebra
(which may not be finite-dimensional), and $\Hom_{{\mathcal
C}_{\cov}^{F,\Phi}}(X, Y)$ is an $\Hom_{{\mathcal C}^{F,\Phi}}(X,
X)$-$\Hom_{{\mathcal C}^{F,\Phi}}(Y, Y)$-bimodule. For more details,
we refer the reader to \cite{Pan, PZ}. In this paper, the category
${\mathcal C}_{\cov}^{F,\Phi}$ is simply called a perforated covering
category, and the algebra $\Hom_{{\mathcal C}^{F,\Phi}}(X, X)$ is
called perforated covering algebra of $X$ without mentioning $F$
and $\Phi$.

In case $\mathcal C$ is the bounded derived category $\Db{A}$ with
$A$ a $k$-algebra, and $F$ is the shift functor [1] of $\Db{A}$, we
denote simply by ${\mathcal G}^{\Phi}_A$ the $(F,\Phi)$-covering
category ${\mathcal C}^{F,\Phi}$, by $\scr G_{A}^{\Phi}(X, Y) $ the set
$\Hom_{{\mathcal G}^{\Phi}_A}(X,Y)$, 

$$ \scr G_{A}^{\Phi}(X, Y)=\Hom_{{\mathcal G}_A^{\Phi}}(X,Y):=\displaystyle \bigoplus_{i,j\in \Phi} \Hom_{\Db{A}}(X[i], Y[j])
=\displaystyle \bigoplus_{i,j\in \Phi} \Hom_{\Db{A}}(X, Y[j-i]),$$
Denote by $\scr G_A^{\Phi}(X)$ the
endomorphism algebra $\Hom_{{\mathcal G}^{\Phi}_A}(X, X)$ of $X$ in
${\mathcal G}^{\Phi}_A$. In this case, each element in
$\scr G_A^{\Phi}(X,Y)$ can be written as $(f_{ij})_{i,j\in\Phi}$ with $f_{ij}\in
\Hom_{\Db{A}}(X, Y[j-i])$. The composition of morphisms in${\mathcal
G}^{\Phi}_A$ can be interpreted as follows: for each triple $(X, Y,
Z)$ of objects in $\Db{A}$,

$$\scr G_{A}^{\Phi}(X, Y) \times \scr G_A^{\Phi}(Y, Z) \lra \scr G_{A}^{\Phi}(X,Z)$$
$$\big((f_{ij})_{u\in \Phi}, (g_{kl})_{v\in\Phi}\big) \mapsto (h_{st})_{s,t\in \Phi},$$
where
$$ h_{st}:=\left\{
\begin{array}{cc}
0& if\ \  0\leq s<t, \\
\sum_{\substack{k-t\in\Phi\\ l-k\in\Phi}}f_{lk}F^{l-k}g_{kt}
&if \ 0\leq t\leq s .
\end{array}
\right.$$
for $s,t\in\Phi$. Clearly, if $\Phi$ is finite, then
$\scr G_A^{\Phi}(X,Y)$ is finite-dimensional for all $X,Y\in A\modcat$.

Let $\Phi$ be a finite admissible subset of $\mathbb{N}$ and let $m=max\{i |i\in\Phi\}$.
Thus we get $m\geq0$. Then the algebras $\scr G_A^{\Phi}
(X)$ for an object $X$ in $\Db{A}$, and mention some basic properties of these algebras.
Firstly, let us define an $R$-module $\scr G_A^{\Phi}
(X)$ as follows:
$$
\scr G_A^{\Phi}
(X)=
\begin{pmatrix}
 \E(X)_{0,0}   & \cdots & \cdots &\E(X)_{0,m-1} & \E(X)_{0,m}\\
 0   &  \E(X)_{1,1}&\cdots & \E(X)_{1,m-1} & \E(X)_{1,m}\\
\vdots & \vdots & \vdots & \vdots & \vdots \\
0& 0 &
\cdots & \E(X)_{m-1,m-1} &\E(X)_{m-1,m}\\
0 &0&\cdots &0 &\E(X)_{m,m}\\
\end{pmatrix}_{\Phi\times\Phi},
$$
where $\E(X)_{i,j}=\Hom_{\Db{A}}(X, X[j-i])$, $s\leq i,j\leq m$. If $j-i\neq \Phi$, then $\E(X)_{i,j}=\Hom_{\Db{A}}(X, X[j-i])=0$. That is, $\scr G_A^{\Phi}(X)=(\E(X)_{i,j})_{s\leq i,j\leq m}$. 

 Recall that  the triangular matrix algebra of this form seems first to appear in the paper \cite{G} by Edward L. Green in 1975. A special case of this kind of algebras appeared in the paper \cite{Bei} by A. A. Beilinson in 1978, where he described the derived category of coherent sheaves over $\mathbb{P}^n$ as the one of this triangular matrix algebra. Perhaps it is more appropriate to name this triangular matrix algebra as the $\Phi$-Beilinson-Green algebra of $X$.

The following lemma tells us that $\scr G_A^{\Phi}(X)$ is an associative ring by the multiplication defined above.

\begin{Lem}\, \label{lem3.1.1}
With the notations above, $\Phi$ is admissible if and only if $\scr G_A^{\Phi}(X)$ is an associative ring.
\end{Lem}

{\bf Proof}. It follows from that $\Phi$ is admissible and the multiplication of $\scr G_A^{\Phi}(X)$ is defined by compositions.  $\square$ 

The following fact of the $\Phi$-Green algebra $\scr G_A^{\Phi}(X)$ is useful, which can be easily check.
Let
$$e_i=
\begin{pmatrix}
0 & 0 & 0& \cdots & 0 & 0\\
0 & 0 & 0& \cdots & 0 & 0\\
\vdots & \vdots & \vdots & 1_{\End_{\Db{A}}(X)} & \vdots & \vdots\\
0   & 0   &  0  & \cdots & 0     & 0\\
0   & 0   &  0  & \cdots & 0     &0\\
\end{pmatrix}.$$ Therefore, $1_{\scr G_A^{\Phi}(X)}=\Sigma^m_{i=s}{e_i}$ is the identity
element of $\scr G_A^{\Phi}(X)$ and as a left $\scr G_A^{\Phi}(X)$-module, $\scr G_A^{\Phi}(X)e_i\cong \E^{\Phi,i}(X),$ where $\E^{\Phi,i}(X)=\oplus_{j=s}^m \E^{i-j,\Phi}_{\Db{A}}(X)$.
Then $\scr G_A^{\Phi}(X)\cong \oplus_{i=s}^m \E^{\Phi,i}
(X)$ as left $\scr G_A^{\Phi}(X)$-modules.

Now, let us state some elementary properties of the Hom-functor
$\scr G^{\Phi}_A(X,-)$.

\begin{Lem}  Suppose that $A$ is an algebra,  that $X$ is
an $A$-module, and that $\Phi$ is a finite admissible subset of
$\mathbb{N}$.

$(1)$ Let $\add_A^{\Phi}(X)$ stand for  the full subcategory of
$\scr G_A^{\Phi}$ consisting of objects in $\add(_AX)$. Then
the Hom-functor \,$\scr G_{A}^{\Phi}(X, -) : \add_A^{\Phi}(X)\lra
\scr G_{A}^{\Phi}(X)\pmodcat $ \,is an equivalence of categories;

$(2)$ Let $B$ be a $k$-algebra, and let $P$ be a $B$-$A$-bimodule
such that $P_A$ is projective. Then there is a canonical algebra
homomorphism $\alpha_P: \scr G^{\Phi}_A(X)\lra \scr G^{\Phi}_B(P\otimes_AX)$
defined by $(f_{ij})_{i,j\in \Phi}\mapsto (P\otimes_Af_{ij})_{i,j\in \Phi}$
for $(f_{ij})_{i,j\in \Phi}\in \scr G^{\Phi}_A(X)$. Thus every left (or
right) $\scr G^{\Phi}_B(P\otimes_AX)$-module can be regarded as a left
(or right) $\scr G^{\Phi}_A(X)$-module via $\alpha_P$. \label{lem1}
\end{Lem}

{\bf Proof}. 
(1) Note that the objects of $\add_A^{\Phi}(X)$ are the same objects as $\add_A(X)$, and $\Hom_{\add_A^{\Phi}(X)}(X, Y)=\scr G_{A}^{\Phi}(X, Y)$.
Then $\scr G_{A}^{\Phi}(X, Y)$ is the Hom-functor $\Hom_{\add_A^{\Phi}(X)}(X, -): \add_A^{\Phi}(X)\lra
\scr G_{A}^{\Phi}(X)\pmodcat $. 
It is known that $$\rad(\scr G_A^{\Phi}(X))=\begin{pmatrix}
\rad(\End_A(X)) &\Ext^1(X,X)&\Ext^2(X,X)&\cdots &\Ext^i(X,X)\\
0&\rad(\End_A(X))&\Ext^1(X,X) &\cdots&\Ext^i(X,X)\\
\vdots&\vdots&\vdots &\vdots&\vdots\\
0&0&0&\cdots&\rad(\End_A(X))
\end{pmatrix},
$$
for some $i\in\Phi$.
Now, it is easy to verify that the functor in (1) is an equivalence of additive categories.

(2) We can check it directly.  $\square$

The following homological result plays an important role in proving
 Theorem \ref{thm1}.

\begin{Lem}
Suppose that  $A, B$ and $C$  are $k$-algebras. Let $_AX$ be an $A$-module, and let $_AY_B$ and ${}_BP_C$ be bimodules with $_BP$
projective. Then, for each $i\geq 0$, we have $\Ext^{i}_{A}(X,
Y\otimes_B P_C)\simeq \Ext^i_A(X, Y)\otimes_B P_C $\, as
$C^{op}$-modules. Moreover, for each  admissible subset $\Phi$ of
$\mathbb{N}$, we have $\scr G^{\Phi}_{A}(X, Y\otimes_B P_C)\simeq
\scr G^{\Phi}_A(X, Y)\otimes_{\scr G_B^{\Phi}(B)}  \begin{pmatrix}
P_C &0&0&\cdots &0\\
0& P_C&0 &\cdots&0\\
\vdots&\vdots&\vdots &\vdots&\vdots\\
0&0&0&\cdots&P_C
\end{pmatrix}$ as
$\scr G^{\Phi}_{A}(X)\mbox{-}\scr G_C^{\Phi}(C)$-bimodules. \label{lem3}
\end{Lem}

{\bf Proof}. The first part of Lemma \ref{lem3} is proved in \cite[Lemma 2.5]{CPX}.

Second, for each admissible subset $\Phi$ of $\mathbb{N}$, we define
a map 
$$\varphi_{\Phi}: \scr G^{\Phi}_A(X, Y)\otimes_{\scr G_B^{\Phi}(B)} \begin{pmatrix}
P_C &0&0&\cdots &0\\
0& P_C&0 &\cdots&0\\
\vdots&\vdots&\vdots &\vdots&\vdots\\
0&0&0&\cdots&P_C
\end{pmatrix} \to
\scr G^{\Phi}_{A}(X, Y\otimes_B P_C)$$ by $(\ol{f_{ij}})\otimes (p_{jj})\mapsto
(\varphi_{ij}(\ol{f_{ij}}\otimes p_{jj}))
$, where $p_{jj}\in P$, and $f_{ij}\in
\Hom{_A}(\Omega_A^{j-i}(X), Y) $ with $j-i\in \Phi$. From the proof of  first part of \cite[Lemma 2.5]{CPX}, we know that  $\varphi_{\Phi}$ is an isomorphism of
$\scr G_C^{\Phi}(C)^{op}$-modules. It suffices to
show that $\varphi_{\Phi}$ is an isomorphism of left
$\scr G_A^{\Phi}(X)$-modules, this follows from the proof of the second part of \cite[Lemma 2.5]{CPX}.
 $\square$

\section{Stable equivalences of Morita type for $\Phi$-Beilinson-Green algebras \label{induction}}

In this section, we shall prove Theorem \ref{thm1}. First, let us
recall the definition of stable equivalences of Morita type in
\cite{B}.

\begin{Def}
Let $A$ and $B$ be (arbitrary) $k$-algebras. We say that $A$ and $B$
are stably equivalent of Morita type if there is an $A$-$B$-bimodule
$_A M_B$ and a $B$-$A$-bimodule $_B N_ A$ such that

$(1)$ $M$ and $N$ are projective as one-sided modules, and

$(2)$  $M\otimes_{B}N\simeq {A\oplus P}$ as $A$-$A$-bimodules for
some projective $A$-$A$-bimodule $P$, and $N\otimes_{A}M \simeq
{B\oplus Q}$ as $B$-$B$-bimodules for some projective
$B$-$B$-bimodule $Q$. \label{stm}
\end{Def}

In this case, we say that $M$ and $N$ define a stable equivalence of
Morita type between  $A$ and $B$. Moreover, we have two exact
functors $T_N:=N\otimes_A- : A\modcat \to B\modcat$ and $T_M
:=M\otimes_B-: B\modcat \to A\modcat$. Similarly, the bimodules $P$
and $Q$ define two exact functors $T_P$ and $T_Q$, respectively.
Note that the images of $T_P$ and $T_Q$ consist of projective
modules.

Let us remark that if $A$ and $B$ have no separable direct summands,
then we may assume that $M$ and $N$ have no non-zero projective
bimodules as direct summands. In fact, If $M=M'\oplus M''$ and
$N=N'\oplus N''$ such that $M'$ and $N'$ have no non-zero projective
bimodules as direct summands, and that $M''$ and $N''$ are
projective bimodules, then it follows from \cite[Lemma 4.8]{LX2}
that $M'$ and $N'$ also define a stable equivalence of Morita type
between $A$ and $B$.

From now on, we assume that $A, B, M, N, P$ and $Q$ are fixed as in
Definition \ref{stm}, and that $X$ is a generator for $A\modcat$.
Moreover, we fix a finite admissible subset $\Phi$ of $\mathbb{N}$,
and define $\Lambda:=\scr G^{\Phi}_{A}(X)$ and
$\Gamma:=\scr G^{\Phi}_{B}(N\otimes_A X)$.

Since the functors $T_N$ and $T_M$ are exact, they preserve
acyclicity, and can be extended to  triangle functors $T_N': \Db
A\to\Db B$ and $T_M': \Db B\to\Db A$, respectively. Furthermore,
$T'_N$ and $T'_M$ induce canonically two functors $F: {\mathcal G}_A^{\Phi}\to {\mathcal G}_B^{\Phi}$ and $G: {\mathcal
G}_B^{\Phi}\to {\mathcal G}_A^{\Phi}$, respectively. More precisely,
if $\cpx{X}\in \Db{A}$, then $F(\cpx{X}):= ( N\otimes_AX^i,
N\otimes_Ad_X^i)$, and if $f:=(f_{i,j})_{i,j\in \Phi}\in \Hom_{{\mathcal
G}^{\Phi}_A}(\cpx{X},\cpx{Y})$ with $\cpx{Y}\in \Db{A}$, then
$F(f):=(N\otimes_Af_{i,j})_{i,j\in \Phi}\in \Hom_{{\mathcal
E}^{\Phi}_B}(F(\cpx{X}),F(\cpx{Y}))$. Similarly, we define the
functor $G$.

The functor $F$ gives rise to a canonical algebra homomorphism
$\alpha_N: \scr G_A^{\Phi}(\cpx{X})\to\scr G_B^{\Phi}(F(\cpx{X}))$ for each
object $\cpx{X}\in \Db{A}$. In particular, for any
$\cpx{Z}\in\Db{B}$, we can regard $\scr G_B^{\Phi}(\cpx{Z}, F(\cpx{X}))$
as  an $\scr G_B^{\Phi}(\cpx{Z})$-$\scr G_A^{\Phi}(\cpx{X})$-bimodule via
$\alpha_N$. Note that  the homomorphism $\alpha_N$ coincides with
the one defined in Lemma \ref{lem1}, when $\cpx{X}$ is an
$A$-module.

\medskip
{\bf Proof of Theorem \ref{thm1}.}  We define $U:=\scr G^{\Phi}_{A}(X,
T_M (N\otimes_A X))$ and $V:=\scr G^{\Phi}_{B}(N\otimes_A X, T_N (X))$.
In the following we shall prove that $U$ and $V$ define a stable
equivalence of Morita type between $\Lambda$ and $\Gamma$.

First, we endow $U$ with a right $\Gamma$-module structure by
$u\cdot \gamma:=uG(\gamma)$ for $u\in U$ and $\gamma\in \Gamma$, and
endow $V$ with a right $\Lambda$-module structure by
$v\cdot\lambda:=vF(\lambda)$ for $v\in V$ and $\lambda\in \Lambda$.
Then, $U$ becomes  a $\Lambda$-$\Gamma$-bimodule, and  $V$ becomes a
$\Gamma$-$\Lambda$-bimodule.

By definition, we know $V=\Gamma$, and it is a projective left
$\Gamma$-module. Note that $_AX$ is a generator and  the images of
$T_P$ consists of projective modules. We conclude  that $T_M
(N\otimes_A X) = M\otimes_B (N\otimes_A X)\simeq X\oplus P\otimes_A
X\in\add(X)$. Thus $U$ is projective as a left $\Lambda$-module by
Lemma \ref{lem1}.

(1) $U \otimes_\Gamma V$, as a $\Lambda$-$\Lambda$-bimodule,
satisfies the condition (2) in Definition \ref{stm}.

Indeed, we write $ W:=\scr G^{\Phi}_{A}(X, (T_MT_N)(X))$, and define a
right $\Lambda$-module structure on $W$ by $w\cdot
\lambda':=w(GF)(\lambda')$ for  $w\in W$ and $\lambda'\in \Lambda$.
Then $W$ becomes a $\Lambda$-$\Lambda$-bimodule. Note  that  there
is a natural $\Lambda$-module isomorphism $\varphi: U \otimes_\Gamma
V \ra W$ defined by $x\otimes y\mapsto xG(y)$ for $x\in U$ and $y\in
V$. We claim that $\varphi$ is an isomorphism of
$\Lambda$-$\Lambda$-bimodules. In fact, it suffices to show that
$\varphi$ respects the structure of right $\Lambda$-modules.
However, this follows immediately from a verification: for $c\in U,
d\in V$ and $a\in \Lambda$, we have
$$
\varphi((c\otimes d)\cdot a)=\varphi(c\otimes (d F(a)))=cG(d
F(a))=cG(d)(GF)(a) =\varphi(c\otimes d)\cdot a.
$$
Combining this bimodule isomorphism $\varphi$ with Lemma \ref{lem1},
we get the following isomorphisms of $\Lambda$-$\Lambda$-bimodules:
 $$ (*)\qquad U \otimes_\Gamma V \simeq
\scr G^{\Phi}_{A}(X, (T_MT_N)(X))\simeq \scr G^{\Phi}_{A}(X, X)\oplus
\scr G^{\Phi}_{A}(X, P\otimes_A X)=\Lambda\oplus \scr G^{\Phi}_{A}(X,
P\otimes_A X),
$$
where  the second isomorphism follows  from  $M\otimes_{B}N\simeq
{A\oplus P}$ as $A$-$A$-bimodules, and where the right
$\Lambda$-module structure on $\scr G^{\Phi}_{A}(X, P\otimes_A X)$ is
induced by the canonical algebra homomorphism 
$$\Lambda \to
\scr G^{\Phi}_{A}( P\otimes_A X),
$$ which sends $(f_{i,j})_{i,j\in \Phi}$ in
$\Lambda$ to $(P\otimes_Af_{i,j})_{i,j\in \Phi}$ (see Lemma \ref{lem1}
(2)).

Now, we show that $\scr G^{\Phi}_{A}(X, P\otimes_A\,X)$ is a projective
$\Lambda$-$ \Lambda$-bimodule. In fact, since $P\in\add(_AA\otimes_k
A_A)$, we conclude that $\scr G^{\Phi}_{A}(X,
P\otimes_A\,X)\in\add(\scr G^{\Phi}_{A}(X, (A\otimes_k A)
\otimes_A\,X))$. Thus, it is sufficient to prove that
 $\scr G^{\Phi}_{A}(X, (A\otimes_k A) \otimes_A\,X)$ is a
projective $\Lambda$-$ \Lambda$-bimodule. For this purpose, we first
note that the right $\Lambda$-module structure on $\scr G^{\Phi}_{A}(X,
(A\otimes_k A) \otimes_A\,X)$ is induced by the canonical algebra
homomorphism  
$$\alpha_{A\otimes_kA}:\Lambda\to
\scr G^{\Phi}_{A}((A\otimes_k A) \otimes_AX),$$ which sends
$g:=(g_{i,j})_{i,j\in \Phi}$ in $\Lambda$ to
$((A\otimes_kA)\otimes_Ag_{i,j})_{i,j\in \Phi}$. Clearly, $_AA\otimes_k
A\otimes_A\,X\in\add(_AA)$. It follows that $\Ext_A^{l}((A\otimes_k
A) \otimes_AX, (A\otimes_k A) \otimes_AX)=0$ for any $l>0$, and
therefore $(A\otimes_kA)\otimes_Ag_{i,j}=0$ for any $0\ne i-j\in\Phi$.
Thus we have $\alpha_{A\otimes_kA}(g)=(A\otimes_kA)\otimes_Ag_{i,i}$. If
$$\pi:\Lambda\to \begin{pmatrix}
\End_A(X) &0&0&\cdots &0\\
0& \End_A(X)&0 &\cdots&0\\
\vdots&\vdots&\vdots &\vdots&\vdots\\
0&0&0&\cdots&\End_A(X)
\end{pmatrix}
$$ is the canonical projection  and $\mu'$ is
the canonical algebra homomorphism 
$$\End_A(X)\ra
\End_A\big((A\otimes_k A)\otimes_AX\big),$$ then
$\alpha_{A\otimes_kA}=\pi \begin{pmatrix}
\mu' &0&0&\cdots &0\\
0& \mu'&0 &\cdots&0\\
\vdots&\vdots&\vdots &\vdots&\vdots\\
0&0&0&\cdots&\mu'
\end{pmatrix}$. Thus the right $\Lambda$-module
structure on $\scr G^{\Phi}_{A}(X, (A\otimes_k A) \otimes_A\,X)$ is
induced by $ \begin{pmatrix}
\End_A(X) &0&0&\cdots &0\\
0& \End_A(X)&0 &\cdots&0\\
\vdots&\vdots&\vdots &\vdots&\vdots\\
0&0&0&\cdots&\End_A(X)
\end{pmatrix}$. Similarly, from the homomorphisms
$$\begin{array}{rl}
\Lambda=\scr G^{\Phi}_{A}(X)\epa {\pi} \begin{pmatrix}
\End_A(X) &0&0&\cdots &0\\
0& \End_A(X)&0 &\cdots&0\\
\vdots&\vdots&\vdots &\vdots&\vdots\\
0&0&0&\cdots&\End_A(X)
\end{pmatrix}\epa { \begin{pmatrix}
\mu &0&0&\cdots &0\\
0& \mu&0 &\cdots&0\\
\vdots&\vdots&\vdots &\vdots&\vdots\\
0&0&0&\cdots&\mu
\end{pmatrix}}\\
\begin{pmatrix}
\End_A((A\otimes_k X)) &0&0&\cdots &0\\
0& \End_A((A\otimes_k X))&0 &\cdots&0\\
\vdots&\vdots&\vdots &\vdots&\vdots\\
0&0&0&\cdots&\End_A((A\otimes_k X))
\end{pmatrix}= \scr G^{\Phi}_{A}(A\otimes_k X),
\end{array}$$ where
$\mu:\End_A(X)\to \End_A(A\otimes_kX)$ is induced by the tensor
functor $A\otimes_k-$, we see that the right $\Lambda$-module
structure on $\scr G^{\Phi}_{A}(X, A\otimes_k X)$ is also induced by
$ \begin{pmatrix}
\End_A(X) &0&0&\cdots &0\\
0& \End_A(X)&0 &\cdots&0\\
\vdots&\vdots&\vdots &\vdots&\vdots\\
0&0&0&\cdots&\End_A(X)
\end{pmatrix}$. Thus $\scr G^{\Phi}_{A}(X, (A\otimes_k A)
\otimes_A\,X)\,\simeq\,\scr G^{\Phi}_{A}(X, A\otimes_k X)$ as
$\Lambda$-$\Lambda$-bimodules. Moreover,  it follows from Lemma
\ref{lem3} that 
$$\scr G^{\Phi}_{A}(X, A\otimes_k X)\,\simeq\,
 \scr G^{\Phi}_{A}(X, A)\otimes  \begin{pmatrix}
X &0&0&\cdots &0\\
0& X&0 &\cdots&0\\
\vdots&\vdots&\vdots &\vdots&\vdots\\
0&0&0&\cdots&X
\end{pmatrix}$$ 
as
$\Lambda$-$ \begin{pmatrix}
\End_A(X) &0&0&\cdots &0\\
0& \End_A(X)&0 &\cdots&0\\
\vdots&\vdots&\vdots &\vdots&\vdots\\
0&0&0&\cdots&\End_A(X)
\end{pmatrix}$-bimodules. Since the $\begin{pmatrix}
X &0&0&\cdots &0\\
0& X&0 &\cdots&0\\
\vdots&\vdots&\vdots &\vdots&\vdots\\
0&0&0&\cdots&X
\end{pmatrix}$ can be
regarded as a right $\Lambda$-module via the homomorphism 
$$\Lambda\epa{\pi} \begin{pmatrix}
\End_A(X) &0&0&\cdots &0\\
0& \End_A(X)&0 &\cdots&0\\
\vdots&\vdots&\vdots &\vdots&\vdots\\
0&0&0&\cdots&\End_A(X)
\end{pmatrix},
$$  we
see that 
$$\scr G^{\Phi}_{A}(X, A\otimes_k X)\,\simeq\,
 \scr G^{\Phi}_{A}(X, A)\otimes_{\scr G_k^{\Phi}(k)} \begin{pmatrix}
X &0&0&\cdots &0\\
0& X&0 &\cdots&0\\
\vdots&\vdots&\vdots &\vdots&\vdots\\
0&0&0&\cdots&X
\end{pmatrix}
$$ as $\Lambda$-$\Lambda$-bimodules and that
 $$ \begin{pmatrix}
X &0&0&\cdots &0\\
0& X&0 &\cdots&0\\
\vdots&\vdots&\vdots &\vdots&\vdots\\
0&0&0&\cdots&X
\end{pmatrix}\simeq\scr G^{\Phi}_{A}(A, X)
$$ as
right $\Lambda$-modules. Thus 
$$ \scr G^{\Phi}_{A}(X, A)\otimes_{\scr G_k^{\Phi}(k)}\begin{pmatrix}
X &0&0&\cdots &0\\
0& X&0 &\cdots&0\\
\vdots&\vdots&\vdots &\vdots&\vdots\\
0&0&0&\cdots&X
\end{pmatrix}\,\in\,(\scr G^{\Phi}_{A}(X, A) \otimes_{\scr G_k^{\Phi}(k)}\scr G^{\Phi}_{A}(A, X)$$ as
$\Lambda$-$\Lambda$-bimodules. Since $_AA\in\add(X)$, we know that
$\scr G^{\Phi}_{A}(X, A)$ is a projective left $\Lambda$-module and
$\scr G^{\Phi}_{A}(A, X)$ is a projective right $\Lambda$-module. Hence
$ \scr G^{\Phi}_{A}(X, A)\otimes_{\scr G_k^{\Phi}(k)} \scr G^{\Phi}_{A}(A, X)$ is a projective
$\Lambda$-$\Lambda$-bimodule. This implies that $\scr G^{\Phi}_{A}(X,
P\otimes_A\,X)$ is a projective $\Lambda$-$ \Lambda$-bimodule.

(2) $V \otimes_\Lambda U$, as a $\Gamma$-$\Gamma$-bimodule, fulfills
the condition (2) in Definition \ref{stm}.

Let $Z:=\scr G^{\Phi}_{B}(N\otimes_A X, T_N T_M(N\otimes_A X))$.
Similarly, we endow $Z$ with a right $\Gamma$-module structure
defined by $z\cdot b:=z (FG)(b)$ for $z\in Z$ and $b\in \Gamma$.
Then $Z$ becomes a $\Gamma$-$\Gamma$-bimodule. Observe that, for
each $A$-module $Y$, there is a homomorphism
 $\Psi_Y: V\otimes_{\Lambda}\scr G^{\Phi}_{A}(X,
Y)\to \scr G^{\Phi}_{B}(N\otimes_A X, T_N(Y))$ of $\Gamma$-modules,
which is defined by $g\otimes h \mapsto gF(h)$ for $g\in V$ and
$h\in \scr G^{\Phi}_{A}(X, Y)$. This homomorphism is natural in $Y$. In
other words, $\Psi:V\otimes_{\Lambda}\scr G^{\Phi}_{A}(X, -)\to
\scr G^{\Phi}_{B}(N\otimes_A X, T_N(-))$  is a natural transformation of
functors from $A\modcat$ to $\Gamma\modcat$. Clearly, $\Psi_X$ is an
isomorphism of $\Gamma$-modules. It follows from $T_M (N\otimes_A
X)\in\add(X)$ that $\Psi{_{T_M (N\otimes_A X)}}: V \otimes_\Lambda U
\to Z$ is a $\Gamma$-isomorphism. Similarly, we can check that
$\Psi{_{T_M (N\otimes_A X)}}$ preserves the structure of right
$\Gamma$-modules. Thus $\Psi{_{T_M (N\otimes_A X)}}: V
\otimes_\Lambda U \to Z$ is an isomorphism of
$\Gamma$-$\Gamma$-bimodules, and there are the following
isomorphisms of $\Gamma$-$\Gamma$-bimodules:
$$ (**) \qquad V \otimes_\Lambda U \simeq Z\simeq
\Gamma \oplus \scr G^{\Phi}_{B}(N\otimes_A X, Q\otimes_B(N\otimes_A X)),
$$ where the second isomorphism is deduced from
$N\otimes_{A}M\simeq {B\oplus Q}$ as $B$-$B$-bimodules. By a
similar argument to that in the proof of (1), we can show that
$\scr G^{\Phi}_{B}(N\otimes_A X, Q\otimes_B(N\otimes_A X))$ is a
projective $\Gamma$-$\Gamma$-bimodule.

It remains to show that $U_{\Gamma}$ and $V_{\Lambda}$ are
projective. This is equivalent to showing that the tensor functors
$T_U:= U\otimes_\Gamma-: \Gamma\modcat \to\Lambda\modcat$ and
$T_V:=V\otimes_\Lambda-: \Lambda\modcat \to\Gamma\modcat$ are exact.
Since tensor functors are always right exact, the exactness of $T_U$
is equivalent to the property that $T_U$ preserve injective
homomorphisms of modules. Now, suppose that $f:C\to D$ is an
injective homomorphism between $\Gamma$-modules $C$ and $D$. Since
$\scr G^{\Phi}_{B}(N\otimes_A X, Q\otimes_B(N\otimes_A X))$ is a right
projective $\Gamma$-module, we know from $(**)$ that the composition
functor $T_VT_U$ is exact. In particular, the homomorphism
$(T_VT_U)(f):(T_VT_U)(C)\to (T_VT_U)(D)$ is injective. Let $\mu:
\Ker(T_U(f))\to T_U(C)$ be the canonical inclusion. Clearly, we have
$\mu T_U(f)=0$, which shows  $T_V(\mu
T_U(f))=T_V(\mu)(T_VT_U)(f)=0$. It follows that  $T_V(\mu)=0$ and
$(T_UT_V)(\mu)=0$. By ($*$), we get $\mu=0$, which implies that the
homomorphism  $T_U(f)$ is injective. Hence $T_U$ preserves injective
homomorphisms. Similarly,  we can show that $T_V$ preserves
injective homomorphisms. Consequently,  $U_\Gamma$ and
$V_\Lambda$ are projective.

Threrfore, the bimodules $U$ and $V$ define a stable equivalence of
Morita type between $\Lambda$ and $\Gamma$. This finishes the proof
of Theorem \ref{thm1}. $\square$

\smallskip
{\it Remarks.} (1) If we take $\Phi=\{0\}$ in Theorem \ref{thm1},
then we get \cite[Theorem 1.1]{LX3}. We mention that the proof of Theorem \ref{thm1} is somewhat routine. 

(2) Since stable equivalences of Morita type preserve the global,
dominant and finitistic dimensions of algebras, Theorem \ref{thm1}
asserts actually also that these dimensions are equal for algebras
$\scr G_A^{\Phi}(X)$ and $\scr G_B^{\Phi}(N\otimes_AX)$.

\section{A family of finite dimensional algebras such that they are derived equivalent but not stably equivalences of Morita type\label{fff}}

\subsection{Stable equivalences of Morita type for self-injective algebras\label{selfinjective}}

Derived equivalences between
self-injective algebras implies stable equivalences of Morita type
by a result of Rickard \cite{R3}, this makes stable equivalences of
Morita type closely  related to the Brou\'e abelian defect group
conjecture which essentially predicates a derived equivalence
between two block algebras \cite{B}, and thus also a stable
equivalence of Morita type between them.

In this section, we will apply Theorem \ref{thm1} and \cite[Theorem 1.2]{CPX} to self-injective algebras. It turns out that the
existence of a stable equivalence of Morita type between
$\Phi$-Beilinson-Green algebras of generators for one finite
admissible set $\Phi$ implies the one for all finite admissible
sets.

Throughout this section, we fix a finite admissible subset $\Phi$ of
$\mathbb{N}$, and assume that  $A$ and $B$ are indecomposable,
non-simple, self-injective algebras. Let $X$ be a generator for
$A$-mod with a decomposition $X:= A\oplus \displaystyle\bz_{1\leq
i\leq n} X_i$ , where $X_i$ is indecomposable and non-projective
such that $X_i\ncong X_t$ for $1\leq i\neq t\leq n$, and let $Y$ be
a generator for $B$-mod with a decomposition $Y:= B\sz
\displaystyle\bz_{1\leq j\leq m} Y_i$, where
 $Y_j$ is indecomposable and non-projective such that $Y_j\ncong
Y_s$ for $ 1\leq j\neq s\leq m$.

\begin{Lem} \label{lem 2.13}
$(1)$ The full subcategory of\, $\scr G_A^{\Phi}(X)\modcat$ consisting
of projective-injective  $\scr G_A^{\Phi}(X)$-modules is equal to
$\add(\scr G_A^{\Phi}(X, A))$. Particularly, if
$\scr G_A^{\Phi}(X)\neq\begin{pmatrix}
\End_A(X) &0&0&\cdots &0\\
0& \End_A(X)&0 &\cdots&0\\
\vdots&\vdots&\vdots &\vdots&\vdots\\
0&0&0&\cdots&\End_A(X)
\end{pmatrix}$, then $\dm(\scr G_A^{\Phi}(X))= 0$.

$(2)$ $\scr G_A^{\Phi}(X)$ has no semisimple direct summands.
\end{Lem}

{\bf Proof}. (1) For convenience, we set $\Lambda= \scr G_A^{\Phi}(X)$.  Since $A$ is self-injective, it follows
from that $\nu{_\Lambda}(\scr G_A^{\Phi}(X,
A))\simeq \scr G_A(X,\nu_AA)\simeq \scr G_A^{\Phi}(X, DA)\in \add\big(
\scr G_A^{\Phi}(X, A)\big)$. Consequently, $\scr G_A^{\Phi}(X, A)$ is a
projective-injective $\Lambda$-module. We claim that, up to
isomorphism, each indecomposable projective-injective
$\Lambda$-module is a direct summand of $\scr G_A^{\Phi}(X, A)$. 
Let $\Lambda'=\begin{pmatrix}
C &M\\
0&D\\
\end{pmatrix}$ be a triangular matrix algebra. Then each $\Lambda'$-module can be represented by a triple $(_CX, _DY, f)$ with
$X\in C\modcat, Y\in C\modcat$ and $f: M_D\otimes Y\to X$ a $C$-module homomorphism. Let  $(_CX, _DY, f)$ be an indecomposable $\Lambda'$-module.
By \cite[Proposition 2.5, p. 76]{ARS}, there are two possibilities:

(i) $_DY=0$ and $_CX$ is an indecomposable projective-injective $C$-module with $M\otimes_CX=0$,

(ii) $_CX=0$ and $_DY$ is an indecomposable projective-injective $D$-module with $\Hom_D(M,Y )=0$.

Thus $\add(\scr G_A^{\Phi}(X, A))$ is just the full subcategory of\,
$\scr G_A^{\Phi}(X)\modcat$ consisting of projective-injective modules.

Finally, we consider the dominant dimension of
$\dm(\scr G_A^{\Phi}(X))$. Suppose $\scr G_A^{\Phi}(X)\neq\begin{pmatrix}
\End_A(X) &0&0&\cdots &0\\
0& \End_A(X)&0 &\cdots&0\\
\vdots&\vdots&\vdots &\vdots&\vdots\\
0&0&0&\cdots&\End_A(X)
\end{pmatrix}$. Since
$A$ is self-injective, we have 
$$\scr G_A^{\Phi}(X, A)=\begin{pmatrix}
\Hom_A(X, A) &0&0&\cdots &0\\
0& \Hom_A(X, A)&0 &\cdots&0\\
\vdots&\vdots&\vdots &\vdots&\vdots\\
0&0&0&\cdots&\Hom_A(X, A)
\end{pmatrix}.$$  It
follows that $\scr G_A^{\Phi}(X, A)$ is annihilated by 
$$\begin{pmatrix}
0 &\Ext^1(X,X)&\Ext^2(X,X)&\cdots &\Ext^m(X,X)\\
0& 0&\Ext^1(X,X)&\cdots&\Ext^{m-1}(X,X)\\
0&0&0&\cdots&\Ext^1(X,X)\\
\vdots&\vdots&\vdots &\vdots&\vdots\\
0&0&0&\cdots&0
\end{pmatrix},$$
but it is not annihilated by $\Lambda$. Hence $\Lambda$ cannot be cogenerated by
$\scr G_A^{\Phi}(X, A)$. This implies that $\dm(\scr G_A^{\Phi}(X))= 0$. 

(2) Contrarily, we suppose that the algebra $\scr G_A^{\Phi}(X)$ has a
semisimple direct summand. Then $\scr G_A^{\Phi}(X)$ has a simple
projective-injective module $S$. According to $(1)$, we know that
$S$ must be a simple projective-injective $\End_A(X)$-module. Then
it follows from the first part of the proof of \cite[Corollary 4.7]{CPX} that $A$ has a semisimple direct summand. Clearly, this
is contrary to our initial assumption that $A$ is indecomposable and
non-simple. Thus $\scr G_A^{\Phi}(X)$ has no  semisimple direct
summands. $\square$

\begin{Theo}\label{coro14}
If the algebras $\scr G_A^{\Phi}(X)$ and $\scr G_B^{\Phi}(Y)$ are stably
equivalent of Morita type, then $n=m$ and  there are bimodules
$_AM_B$ and $_BN_A$ which define a stable equivalence of Morita type
between $A$ and $B$ such that, up to the ordering of indices,
\,$_AM\otimes_B Y_i\simeq {X_i\oplus P_i}$ as $A$-modules, where
$_AP_i$ is projective for all $i$ with $ 1\leq i\leq n$. Moreover,
for any finite admissible subset $\Psi$ of $\mathbb{N}$, there is a
stable equivalence of Morita type between $\scr G_A^{\Psi}(X)$ and
$\scr G_B^{\Psi}(Y)$.
\end{Theo}

{\bf Proof}. For convenience, we set $\Lambda_0=\End(X)$, $\Lambda= \scr G_A^{\Phi}(X)$ and $\Gamma_0=\End(Y), \Gamma=
\scr G_B^{\Phi}(Y)$. By Lemma \ref{lem 2.13}, the algebras $\Lambda$ and
$\Gamma$ have no semisimple direct summands.  Let $e$ be the
idempotent in $\Lambda_{0}$ corresponding to the direct summand $A$
of $X$, and let $f$ be the idempotent in $\Gamma_{0}$ corresponding
to the direct summand $B$ of $Y$. Let $e'=\begin{pmatrix}
e &0&0&\cdots &0\\
0& e&0 &\cdots&0\\
\vdots&\vdots&\vdots &\vdots&\vdots\\
0&0&0&\cdots&e
\end{pmatrix}$
 and 
 $f'=\begin{pmatrix}
f &0&0&\cdots &0\\
0& f&0 &\cdots&0\\
\vdots&\vdots&\vdots &\vdots&\vdots\\
0&0&0&\cdots&f
\end{pmatrix}$ be the idempotents of $\Lambda$ and $\Gamma$, respectively.
Note that $\Lambda e'\simeq
\scr G_A^{\Phi}(X, A)$ as $\Lambda$-modules  and $\Gamma  f'\simeq
\scr G_B^{\Phi}(Y, B)$ as $\Gamma$-modules. Clearly, $e'\Lambda e'\simeq
\begin{pmatrix}
A &0&0&\cdots &0\\
0& A&0 &\cdots&0\\
\vdots&\vdots&\vdots &\vdots&\vdots\\
0&0&0&\cdots&A
\end{pmatrix}$ and $f'\, \Gamma f' \simeq \begin{pmatrix}
B &0&0&\cdots &0\\
0& B&0 &\cdots&0\\
\vdots&\vdots&\vdots &\vdots&\vdots\\
0&0&0&\cdots&B
\end{pmatrix}$ as algebras. Moreover, we see that
$e'\Lambda \simeq \begin{pmatrix}
X &0&0&\cdots &0\\
0& X&0 &\cdots&0\\
\vdots&\vdots&\vdots &\vdots&\vdots\\
0&0&0&\cdots&X
\end{pmatrix}$ as $e'\Lambda e'$-modules, and $f'\, \Gamma \simeq \begin{pmatrix}
Y &0&0&\cdots &0\\
0& Y&0 &\cdots&0\\
\vdots&\vdots&\vdots &\vdots&\vdots\\
0&0&0&\cdots&Y
\end{pmatrix}$ as
$f'\Gamma f'$-modules. Suppose that a stable equivalences of Morita type
between $\Lambda$ and $\Gamma$ is given.  we know that the idempotent $e'$ in $\Lambda$
and the idempotent $f'$ in $\Gamma$ satisfy the conditions in \cite[Theorem 1.2]{CPX}. It follows from  that there are bimodules $_{e'\Lambda e'}M'_{f'\Gamma f'}$ and 
$_{f'\Gamma f'}N'_{e'\Lambda e'}$ which
define a stable equivalence of Morita type between $e'\Lambda e'$ and $f'\Gamma f'$ such
that $\add(M'\otimes_{f'\Gamma f'} f'\Gamma)= \add(e'\Lambda)$. 

We know that $$e'\Lambda e'\simeq
\begin{pmatrix}
A &0&0&\cdots &0\\
0& A&0 &\cdots&0\\
\vdots&\vdots&\vdots &\vdots&\vdots\\
0&0&0&\cdots&A
\end{pmatrix}\simeq A\times A\times\cdots\times A $$ and $$f'\, \Gamma f' \simeq \begin{pmatrix}
B &0&0&\cdots &0\\
0& B&0 &\cdots&0\\
\vdots&\vdots&\vdots &\vdots&\vdots\\
0&0&0&\cdots&B
\end{pmatrix}\simeq\times B\times B\times \cdots\times B$$ as algebras. By \cite[Theorem 2.2]{L1} or the proof of \cite[Lemma 4.1]{CPX}, there are stable equivalences of Morita type between summands of $e'\Lambda e'$ and $f'\Gamma f'$.
 It follows from  that there are bimodules $_AM_B$ and 
$_BN_A$ which
define a stable equivalence of Morita type between $A$ and $B$ such
that $\add(M\otimes_B Y)= \add(X)$. 
By the given decompositions of
$X$ and $Y$, we conclude that $n=m$ and, up to the ordering of
direct summands, we may assume that \,$_AM\otimes_B Y_i\simeq
{X_i\oplus P_i}$ as $A$-modules, where $_AP_i$ is projective for all
$i$ with $ 1\leq i\leq n$. Now, the last statement in this corollary
follows immediately from Theorem \ref{thm1}. Thus the proof is
completed. $\square$

\medskip
Usually, it is difficult to decide whether an algebra is not stably
equivalent of Morita type to another algebra. The next corollary,
however, gives a sufficient condition to assert when two algebras
are not stably equivalent of Morita type.

\begin{Koro}\label{coro15}
Let $n$ be a non-negative integer. Let $W$ be an indecomposable
non-projective $A$-module. Suppose that $\Omega_A^{s}(W)\not\simeq
W$ for any non-zero integer $s$. Set $W_n= \bz _{0\leq i\leq n}
\Omega_A^{i}(W)$. Then, for any finite admissible subset $\Psi$ of
$\mathbb{N}$,  the algebras $\scr G_A^{\Psi}(A \oplus W_n \oplus
\Omega_A^{l}(W))$ and $\scr G_A^{\Psi}(A \oplus W_n\oplus
\Omega_A^{m}(W))$ are not stably equivalent of Morita type whenever
$m $ and $l$ belong to $\mathbb{N}$ with $n< m< l$.
\end{Koro}
{\bf{Proof}}. Suppose that there is  a finite admissible subset
$\Psi$ of $\mathbb{N}$ such that $\scr G_A^{\Psi}(A \oplus W_n \oplus
\Omega_A^{m}(W))$ and $\scr G_A^{\Psi}(A \oplus W_n\oplus
\Omega_A^{l}(W))$ are  stably equivalent of Morita type for some
fixed $l, m\in\mathbb{N}$ with $n< m< l$. Set $\Phi_1=\{0,1,\cdots,
n\}\cup\{l\}$ and $\Phi_2=\{0,1,\cdots,n\}\cup\{ m\}$. Then, by
Theorem \ref{coro14}, we know that there exist bimodules $_AM_A$ and
$_AN_A$ which define a stable equivalence of Morita type between $A$
and itself, and  that there is a bijection $\sigma: \Phi_1\to
\Phi_2$ such that $M \otimes_A \Omega_{A}^{j}(W) \simeq
\Omega{_{A}^{\sigma(j)}}(W) \oplus P_j$ as $A$-modules, where $P_j$
is  projective  for each  $j\in \Phi_1$. In particular, we have $M
\otimes_A W \simeq \Omega{_{A}^{\sigma(0)}}(W) \oplus P_0$. Since
$M$ is  projective as a one-sided module, we know that
$M\otimes_A\Omega_{A}^{l}(W) \simeq \Omega{_{A}^{\sigma(0)+l}}(W)
\oplus P_l'$ with $P_l' \in \add(_AA)$. Note that  $M \otimes_A
\Omega_{A}^{l}(W) \simeq \Omega{_{A}^{\sigma(l)}}(W) \oplus P_l$. It
follows that $\Omega{_{A}^{\sigma(0)+l}}(W) \simeq
\Omega{_{A}^{\sigma(l)}}(W) $. Consequently, we have $\sigma(l)=
\sigma(0)+l \geq l$ since $W$ is not $\Omega$-periodic. Hence $l\le
\sigma(l)\leq m < l$, a contradiction. This shows that
$\scr G_A^{\Psi}(A \oplus W_n \oplus \Omega_A^{m}(W))$ and
$\scr G_A^{\Psi}(A \oplus W_n\oplus \Omega_A^{l}(W))$ cannot be stably
equivalent of Morita type whenever $l$ and $m$ $\in\mathbb{N}$ with
$n< m < l$. $\square$

\medskip
This corollary will be used in the next subsection.

\subsection{A family of derived-equivalent algebras: another answer to Thorsten Holm's question \label{family}}

In a talk at a workshop in Goslar, Germany, Thorsten Holm considered
the following qustion on derived equivalences and stable
equivalences of Morita type:

\medskip
{\bf Question.} Is there any infinite series of finite-dimensional
$k$-algebras  such that they have the same dimension and are all
derived-equivalent, but not stably equivalent of Morita type ?

\medskip

It has an affirmative answer in \cite{CPX}. In this section, we shall apply our results in the previous sections
to give an affirmative answer for question again.

Liu and Schulz in \cite{Ls} constructed a local
symmetric $k$-algebra $A$ of dimension 8 and an indecomposable
$A$-module $M$ such that all the syzygy modules $\Omega_A^n (M)$
with $n\in \mathbb{Z}$ are 4-dimensional and pairwise
non-isomorphic.  This algebra $A$ depends on a non-zero parameter
$q\in k$, which is not a root of unity, and has an infinite
DTr-orbit in which each module has the same dimension. Ringel in \cite{R} carried out a thorough
investigation of Auslander-Reiten components of this algebra. Based on this symmetric algebra
and a recent result in \cite{Pan} together with the results in the
previous sections, we shall construct an infinite family of
algebras, which provides a positive solution to the above question.

From now on, we fix a non-zero element $q$ in the field $k$, and
assume that $q$ is not a root of unity. The $k$-algebra $A$ defined by Liu-Schulz is an associative algebra
(with identity) over $k$ with the generators: $x_0, x_1, x_2$, and
the relations: $ x_i^2= 0, \quad \mbox{and}\quad x_{i+1}x_i+q
x_ix_{i+1}= 0 \quad \mbox{for}\quad  i= 0, 1, 2. $

\noindent Here, and in what follows, the subscript is modulo $3$.

Let $n$ be a fixed natural number, and let $\Phi=\{0\}$ or $\{0,
1\}$. For $j\in \mathbb{Z}\,$, set $u_j: =x_2+q^jx_1 $, $I_j :=
Au_j$, $J_j:=u_jA$, $I:= \bz\limits_{i=0}^{n}I_i$ and
$\Lambda_j^{\Phi}:= \scr G_A^{\Phi}(A \sz I \sz I_j )$.

With these notations in mind, the main result in this section can be
stated as follows:

\begin{Theo} For any  $m \geq n+4$ and $\Phi=\{0,
1\}$,  we have

\medskip
$(1)$ $\dk{(\Lambda_m^{\Phi})}= \dk{(\Lambda_{m+1}^{\Phi})}$.

\medskip
$(2)$ $\gd(\Lambda_m^{\Phi})= \infty$.

\medskip
$(3)$ $\dm(\Lambda_m^{\Phi})=0.$

\medskip
$(4)$  $\Lambda_m^{\Phi}$ and $\Lambda_{m+1}^{\Phi}$ are
derived-equivalent.

\medskip
$(5)$ If $l>m$, then $\Lambda_l^{\Phi}$ and $\Lambda_m^{\Phi}$ are
not stably equivalent of Morita type. \label{th1}
\end{Theo}

\begin{Rem}  Note that if $\Phi=\{0\}$, then the above result was known in
\cite{CPX}.
\end{Rem}

The following corollary is a consequence of Theorem \ref{th1}, which solves the above mentioned question positively.

\begin{Koro}
There exists an infinite series of finite-dimensional $k$-algebras
$A_i,\,i\in\mathbb{N}$, such that

$(1)$ $\dk(A_i)= \dk(A_{i+1})$ for all $i\in {\mathbb N}$,

$(2)$ all $A_i$ have the same global and dominant dimensions,

$(3)$ all $A_i$ are derived-equivalent, and

$(4)$ $A_i$ and $A_j$ are not stably equivalent of Morita type for
$i\neq j$. \label{coro16}
\end{Koro}

The proof of Theorem \ref{th1} will cover the rest of this section.

\medskip
The following  result can be directly deduced from the work of in \cite{Pan, PZ}.

\begin{Lem}\label{lem5}
Let $B$ be a $k$-algebra. Let   $Y$ and $M$ be $B$-modules with $M$
a generator for $B\modcat$. If $\Ext_B^1(M, \Omega_B(Y))=
\Ext_B^1(Y, M)=0$, then the endomorphism algebras $\End_B(M\sz Y)$
and $\End_B(M\sz \Omega_B(Y))$  are derived equivalent.  If, in
addition, $\Ext_B^2(M, \Omega_B(Y))= \Ext_B^2(Y, M)=0$, then the
$\{0, 1\}$-Auslander-Yoneda algebras $ \scr G_B^{\{0,1\}}(M\sz Y)$ and
$\scr G_B^{\{0,1\}}(M\sz \Omega_B(Y))$ are derived equivalent.

\end{Lem}

Having made the previous preparations, now we can prove Theorem
\ref{th1}.

{\bf Proof of Theorem \ref{th1}}. Let $m\geq n+4$. Set $M :=\,A\sz
I$ with $I=\bz\limits_{i=0}^{n}I_i$, and $V_m: = M \oplus I_m$.

(1) By \cite[Lemma 6.5(5)]{CPX}, we know that  $\Ext_A^1(M, I_m)=
\Ext_A^1(I_m, M)= 0$. Clearly, we have
$$\dk(\Lambda_m^{\{0\}})=\dk\End_A(M)+\dk\Hom_A(M,I_m)+\dk\Hom_A(I_m,
M)+\dk\End_A(I_m) $$and
$$\dk(\Lambda_m^{\{0,1\}})= 2\dk(\Lambda_m^{\{0\}})+ \dk\Ext_A^1(M, M)+
\dk\Ext_A^1(I_m, I_m).$$ By \cite[Lemma 6.5]{CPX},   we get
$$\dk\End_A(I_m)=3,\; \dk\Ext_A^1(I_m, I_m)=1,
\;\dk\Hom_A(M,I_m)= \dk\Hom_A(I_m, M)=2n+6.$$  It follows that
$\dk{(\Lambda_m^{\Phi})}= \dk{(\Lambda_{m+1}^{\Phi})}$.

\smallskip
(2) It follows from \cite[Theorem 6.1]{CPX} that  $\gd(\Lambda_m^{\{0\}})=\infty$. We know that 
$$\Lambda_m^{\{0,1\}}=\begin{pmatrix}
\Lambda_m^{\{0\}} &\Ext^1(M\oplus I_m)\\
0&\Lambda_m^{\{0\}}\\
\end{pmatrix}$$
This yields $\gd(\Lambda_m^{\{0,1\}})=\infty$.

\smallskip
(3) By Lemma \ref{lem 2.13}, we have $\dm(\Lambda_m^{\{0,1\}})= 0$.

(4) Consider the exact sequence
$$
\delta_m: 0\lra I_{m+1}\lra A\lra I_m\lra 0$$ in $ A\modcat$. Since
$m\geq n+4$, it follows from \cite[Lemmas 6.5(5) and 6.4(4)]{CPX} that $\Ext_A^1(M, I_{m+1})= \Ext_A^1(I_{m+1}, M)=
\Ext_A^1(I_m, M)= \Ext_A^1(M, I_{m})= 0$. Note that $A$ is
self-injective. By Lemma \ref{lem5}, we conclude that the algebras
$\Lambda_m^{\Phi}$ and $\Lambda_{m+1}^{\Phi}$ are derived-equivalent
for $\Phi=\{0, 1\}$.

(5) It follows from \cite[Lemma 6.4]{CPX} that  $\Omega_A(I_j)=I_{j+1}$
for each $j\in \mathbb{Z}$ and that the $A$-modules $I_j$  are
pairwise non-isomorphic for all $j\in \mathbb{Z}$. Now, we define
$W:= I_0$ and $W_n:=\oplus_{0\le j\le n} I_j$. Then, by Corollary
\ref{coro15}, the algebras $\Lambda_l^{\Phi}$ and $\Lambda_m^{\Phi}$
are not stably equivalent of Morita type if $l>m$. Thus the proof is
completed. $\square$

\medskip

\section{Stable equivalences of Morita type for Beilinson-Greeen algebras of associated graded algebras\label{BG}}

In this section, we can construct stable equivalences of Morita type between Beilinson-Green algebra of 
group graded algebra from group graded stable equivalences of Morita type.

Let $A=A_0\oplus A_1\oplus A_2\cdots \oplus A_n$ be a graded $k$-algebra with multiplication induced from multiplication of $A$, given by
$$
A_i\otimes A_j\to \left\{
\begin{array}{cc}
A_{i+j}& if\ \  i+j\leq n \\
0
&if \ i+j\geq n .
\end{array}
\right.$$
We say a graded $A$-module $M$ has degree $\leq n$ if $M_i=0$ for $i\geq n+1$. Let $\G_n(A)$ be the category of graded $A$-modules with degree $\leq n$ and degree zero maps.
Recall that Beilinson-Green algebra \cite{Bei, G} is defined as follows:
$$
\bar{A}=\begin{pmatrix}
 A_0   & A_1& \cdots &A_{n-1} &A_n\\
 0   &  A_0&\cdots & A_{n-2} & A_{n-1}\\
\vdots & \vdots & \vdots & \vdots & \vdots \\
0& 0 &
\cdots &A_0 &A_1\\
0 &0&\cdots &0 &A_0\\
\end{pmatrix},
$$
Let $M\in\G_n(A\otimes B^{op})$ and $N\in\G_n(B\otimes A^{op})$.
Then we have the main theorem of this section. 

\begin{Theo} Suppose that $M$ and $N$ induce a graded stable equivalence of Morita type between $A$ and $B$, then there is a stable equivalence of Morita type between the Beilinson-Green algebras $\bar{A}$ and $\bar{B}$.\label{pp}
\end{Theo}

{\bf{Proof}}. 
Suppose that $M$ and $N$ induce a graded stable equivalence of Morita type between $A$ and $B$, then 

$(1)$ $M$ is a left graded projective $A$-module and right graded projective $B$-module, and $N$ is a left graded projective $B$-module and right graded projective $A$-module.

$(2)$ $M\otimes_{B}N\simeq {A\oplus P}$ as $A$-$A$-graded bimodules for
some projective graded $A$-$A$-bimodule $P$, and $N\otimes_{A}M \simeq
{B\oplus Q}$ as $B$-$B$-graded bimodules for some projective graded 
$B$-$B$-bimodule $Q$. 

Set $
\bar{M}=\begin{pmatrix}
 M_0   & M_1 & \cdots &M_{n-1} &M_n\\
 0   &  M_0&\cdots & M_{n-2} & M_{n-1}\\
\vdots & \vdots & \vdots & \vdots & \vdots \\
0& 0 &
\cdots &M_0 &M_1\\
0 &0&\cdots &0 &M_0\\
\end{pmatrix} \quad\text{and}\quad \bar{N}=\begin{pmatrix}
 N_0   & N_1 & \cdots &N_{n-1} &N_n\\
 0   &  N_0&\cdots & N_{n-2} & N_{n-1}\\
\vdots & \vdots & \vdots & \vdots & \vdots \\
0& 0 &
\cdots &N_0 &N_1\\
0 &0&\cdots &0 &N_0\\
\end{pmatrix}.$
Since 
$$
\begin{array}{rl}
\bar{A}\bar{M}\bar{B}=
\begin{pmatrix}
 A_0   & A_1& \cdots &A_{n-1} &A_n\\
 0   &  A_0&\cdots & A_{n-2} & A_{n-1}\\
\vdots & \vdots & \vdots & \vdots & \vdots \\
0& 0 &
\cdots &A_0 &A_1\\
0 &0&\cdots &0 &A_0\\
\end{pmatrix}
\begin{pmatrix}
 M_0   & M_1 & \cdots &M_{n-1} &M_n\\
 0   &  M_0&\cdots & M_{n-2} & M_{n-1}\\
\vdots & \vdots & \vdots & \vdots & \vdots \\
0& 0 &
\cdots &M_0 &M_1\\
0 &0&\cdots &0 &M_0\\
\end{pmatrix} 
\begin{pmatrix}
 B_0   &B_1 & \cdots &B_{n-1} &B_n\\
 0   &  B_0&\cdots & B_{n-2} & B_{n-1}\\
\vdots & \vdots & \vdots & \vdots & \vdots \\
0& 0 &
\cdots &B_0 &B_1\\
0 &0&\cdots &0 &B_0\\
\end{pmatrix}\\=\begin{pmatrix}
 A_0\otimes M_0\otimes B_0  & A_0\otimes M_0\otimes B_1\oplus A_0\otimes M_1\otimes B_0\oplus A_1\otimes M_0\otimes B_0 & \cdots &\oplus_{i+j+k=n} A_i\otimes M_j\otimes B_k\\
 0   &  A_0\otimes M_0\otimes B_0&\cdots & \oplus_{i+j+k=n-1} A_i\otimes M_j\otimes B_k\\
\vdots & \vdots & \vdots & \vdots  \\
0& 0 &
\cdots &  \oplus_{i+j+k=1} A_i\otimes M_j\otimes B_k\\
0 &0&\cdots &A_0\otimes M_0\otimes B_0\\
\end{pmatrix},
\end{array}
$$
we thus get $\bar{A}\bar{M}\bar{B}\subset\bar{M}$, so $\bar{M}$ is a $\bar{A}$-$\bar{B}$-bimodule. Similarly, $\bar{N}$ is a $\bar{B}$-$\bar{A}$-bimodule. 

$(1)$ We claim that $\bar{M}$ is a left projective $\bar{A}$-module and right projective $\bar{B}$-module. It is well-known that $\{A_{n-i}(-i)|0\leq i\leq n\}$ are graded projective $A$-modules.
It follows from $_AM\in \add \oplus_{0\leq i\leq n} A_{n-i}(-i)$ that $_{\bar{A}}\bar{M}=\oplus_{0\leq i\leq n} M_{n-i}(-i)\in \add \bar{A}$. 
Consequently,  $\bar{M}$ is a left projective $\bar{A}$-module. Similarly, $\bar{M}$ is a  right projective $\bar{B}$-module and $\bar{N}$ is a left projective $\bar{B}$-module and right projective $\bar{A}$-module.

$(2)$ $
\bar{P}=\begin{pmatrix}
 P_0   & P_1 & \cdots &P_{n-1} &P_n\\
 0   &  P_0&\cdots & P_{n-2} & P_{n-1}\\
\vdots & \vdots & \vdots & \vdots & \vdots \\
0& 0 &
\cdots &P_0 &P_1\\
0 &0&\cdots &0 &P_0\\
\end{pmatrix} $ 
 and 
 $
\bar{Q}=\begin{pmatrix}
 Q_0   & Q_1 & \cdots &Q_{n-1} &Q_n\\
 0   &  Q_0&\cdots & Q_{n-2} & Q_{n-1}\\
\vdots & \vdots & \vdots & \vdots & \vdots \\
0& 0 &
\cdots &Q_0 &Q_1\\
0 &0&\cdots &0 &Q_0\\
\end{pmatrix} $ 
are projective bimodules.
If $P$ is a graded projective $A$-$A$-bimodule, then $P\in \add (A\otimes A^{op})$ in $\gr{A\otimes A^{op}}$. Since
$$
\begin{array}{rl}
\bar{A}\otimes\bar{A}=
\begin{pmatrix}
 A_0   & A_1 & \cdots &A_{n-1} &A_n\\
 0   &  A_0&\cdots & A_{n-2} & A_{n-1}\\
\vdots & \vdots & \vdots & \vdots & \vdots \\
0& 0 &
\cdots &A_0 &A_1\\
0 &0&\cdots &0 &A_0\\
\end{pmatrix}\otimes 
\begin{pmatrix}
 A_0   & A_1 & \cdots &A_{n-1} &A_n\\
 0   &  A_0&\cdots & A_{n-2} & A_{n-1}\\
\vdots & \vdots & \vdots & \vdots & \vdots \\
0& 0 &
\cdots &A_0 &A_1\\
0 &0&\cdots &0 &A_0\\
\end{pmatrix}=\\\begin{pmatrix}
 A_0\otimes A_0 & A_0\otimes A_1\oplus A_1\otimes A_0& \cdots &\oplus_{i+j=n} A_i\otimes A_j\\
 0   &  A_0\otimes B_0&\cdots & \oplus_{i+j+k=n-1} A_i\otimes A_j\\
\vdots & \vdots & \vdots & \vdots  \\
0& 0 &
\cdots & A_0\otimes A_1\oplus A_1\otimes A_0\\
0 &0&\cdots &A_0\otimes B_0\\
\end{pmatrix}
\end{array},
$$
it follows that $\bar{P}=\begin{pmatrix}
 P_0   & P_1 & \cdots &P_{n-1} &P_n\\
 0   &  P_0&\cdots & P_{n-2} & P_{n-1}\\
\vdots & \vdots & \vdots & \vdots & \vdots \\
0& 0 &
\cdots &P_0 &P_1\\
0 &0&\cdots &0 &P_0\\
\end{pmatrix}\in \add(\bar{A}\otimes\bar{A}^{op})$. Then $\bar{P}$ is a projective $\bar{A}$-$\bar{A}$-bimodule. 
Similarly, $\bar{Q}$ is a projective $\bar{B}$-$\bar{B}$-bimodule.

$(3)$ $$
\begin{array}{rl}
\bar{M}\otimes\bar{N}=\begin{pmatrix}
 M_0   & M_1 & \cdots &M_{n-1} &M_n\\
 0   &  M_0&\cdots & M_{n-2} & M_{n-1}\\
\vdots & \vdots & \vdots & \vdots & \vdots \\
0& 0 &
\cdots &M_0 &M_1\\
0 &0&\cdots &0 &M_0\\
\end{pmatrix} \otimes\begin{pmatrix}
 N_0   & N_1 & \cdots &N_{n-1} &N_n\\
 0   &  N_0&\cdots & N_{n-2} & N_{n-1}\\
\vdots & \vdots & \vdots & \vdots & \vdots \\
0& 0 &
\cdots &N_0 &N_1\\
0 &0&\cdots &0 &N_0\\
\end{pmatrix}\\\simeq
\begin{pmatrix}
 M_0\otimes N_0 & M_0\otimes N_1\oplus M_1\otimes N_0& \cdots &\oplus_{i+j=n} M_i\otimes N_j\\
 0   &  M_0\otimes N_0&\cdots & \oplus_{i+j+k=n-1} M_i\otimes N_j\\
\vdots & \vdots & \vdots & \vdots  \\
0& 0 &
\cdots & M_0\otimes N_1\oplus M_1\otimes N_0\\
0 &0&\cdots &M_0\otimes N_0\\
\end{pmatrix}\\\simeq
\begin{pmatrix}
 A_0\oplus P_0  & A_1\oplus P_1& \cdots &A_{n-1}\oplus P_{n-1} &A_n\oplus P_n\\
 0   &  A_0\oplus P_0&\cdots & A_{n-2}\oplus P_{n-2}& A_{n-1}\oplus P_{n-1}\\
\vdots & \vdots & \vdots & \vdots & \vdots \\
0& 0 &
\cdots &A_0\oplus P_0 &A_1\oplus P_1\\
0 &0&\cdots &0 &A_0\oplus P_0\\
\end{pmatrix}\\\simeq
\begin{pmatrix}
 A_0   & A_1 & \cdots &A_{n-1} &A_n\\
 0   &  A_0&\cdots & A_{n-2} & A_{n-1}\\
\vdots & \vdots & \vdots & \vdots & \vdots \\
0& 0 &
\cdots &A_0 &A_1\\
0 &0&\cdots &0 &A_0\\
\end{pmatrix}
\oplus
\begin{pmatrix}
 P_0   & P_1 & \cdots &P_{n-1} &P_n\\
 0   &  P_0&\cdots & P_{n-2} & P_{n-1}\\
\vdots & \vdots & \vdots & \vdots & \vdots \\
0& 0 &
\cdots &P_0 &P_1\\
0 &0&\cdots &0 &P_0\\
\end{pmatrix}\simeq\bar{A}\oplus\bar{P}.
\end{array}
$$ as $\bar{A}$-$\bar{A}$-bimodules. 
Similarly, $\bar{N}\otimes\bar{M}\simeq\bar{B}\oplus\bar{Q}$ as $\bar{B}$-$\bar{B}$-bimodules
This completes the proof of Theorem \ref{pp}. $\square$

By the above methods, we can get stable equivalences of Morita type between the Beilinson-Green algebras from $G$-graded stable equivalences of Morita type between $G$-graded algebras $A$ and $B$ for a finite group $G$.
Associated a $G$-graded $k$-algebra $A$, there is a Beilinson-Green algebra $\bar{A}$ defined as a $G\times G$-matrix algebra with $(\bar{A}_{gh})_{g,h\in G}$, where 
$\bar{A}_{gh}=A_{gh^{-1}}$. If $M=\oplus_{g\in G}M_g$ is a $G$-graded $A$-$B$-bimodule, then $\bar{M}=(M_{gh^{-1}})_{g,h\in G}$ is a $\bar{A}$-$\bar{B}$-bimodule.
Hence we have the following theorem.

\begin{Theo} Suppose that there is a $G$-graded stable equivalence of Morita type between $G$-graded algebras $A$ and $B$, then there is a stable equivalence of Morita type between the Beilinson-Green algebras $\bar{A}$ and $\bar{B}$.
\end{Theo}

{\bf{Proof}}.  Suppose that $M$ and $N$ induce a $G$-graded stable equivalence of Morita type between $A$ and $B$, then 

$(1)$ $M$ is a left graded projective $A$-module and right graded projective $B$-module, and $N$ is a left graded projective $B$-module and right graded projective $A$-module.

$(2)$ $M\otimes_{B}N\simeq {A\oplus P}$ as graded $A$-$A$ bimodules for
some graded projective $A$-$A$-bimodule $P$, and $N\otimes_{A}M \simeq
{B\oplus Q}$ as graded $B$-$B$ bimodules for some graded projective  
$B$-$B$-bimodule $Q$. 
Then we can prove that $bar{M}$ is projective $\bar{A}$-module and right projective $\bar{B}$-module and $\bar{N}$ is a left projective $\bar{B}$-module and right projective $\bar{A}$-module. We have the following 
$$
\begin{array}{rl}
\bar{M}\otimes_{\bar{A}}\bar{N}=(M_{gh^{-1}})_{g,h\in G}\otimes(N_{uv^{-1}})_{u,v\in G}
=(\oplus_{k\in G}M_{ak^{-1}}\otimes N_{kb^{-1}})_{a,b\in G}\simeq (A_{ab^{-1}}\oplus P_{ab^{-1}})_{a,b\in G}\\
\simeq\bar{A}\oplus\bar{P}.
\end{array}
$$
Similarly, $\bar{N}\otimes_{\bar{B}}\bar{M}
\simeq\bar{B}\oplus\bar{Q}$.

Secondly, we can prove that $\bar{P}$ is a projective $\bar{A}$-$\bar{A}$-bimodule and $\bar{Q}$ is a projective $\bar{B}$-$\bar{B}$-bimodule.
Then $\bar{M}$ and $\bar{N}$ define a stable equivalence of Morita type between $\bar{A}$ and $\bar{B}$. $\square$

\section{An example\label{Example}}
In this section, we will give an example to illustrate our main result. 

Let $A$ be a self-injective algebra. Then it is well-known that $\Omega_{A\otimes A^{op}}(A)$ defines a stable equivalence of Morita type between $A$ and itself (see \cite{LX2}). 

The following corollary is a consequence of Theorem \ref{thm1}.

\begin{Koro}\label{7.1}
Let $A$ be finite-dimensional self-injective $k$-algebras over a field $k$ and let $X$ be an
$A$-module. Then, for
any finite admissible set $\Phi$ of natural numbers, there is a
stable equivalence of Morita type between $\scr G^{\Phi}_{A}(A\oplus X)$ and
$\scr G^{\Phi}_{A}(A\oplus \Omega_A(X))$.
\end{Koro}

\begin{proof} Note that $\Omega_{A\otimes A^{op}}(A)\otimes X\simeq X\oplus U$, where $U$ is a projective $A$-module. Then the result follows from Theorem \ref{thm1}.
\end{proof}

\begin{Rem} By Lemma \ref{lem5}, $\scr G^{\Phi}_{A}(A\oplus X)$ and
$\scr G^{\Phi}_{A}(A\oplus \Omega_A(X))$ are also derived equivalent.
\end{Rem}

Many important classes of algebras are of the form $\End_A(A\oplus
X)$ with $A$ a self-injective algebra. From the above corollary (see
also \cite[Corollary 3.14]{HAU}), we may get a series of algebras
which are stably equivalent of Morita type to Schur algebras. For
unexplained terminology in the next corollary, we refer the reader
to \cite{Gr}.

\begin{Koro} Suppose that $k$ is an algebraically closed field. Let $S_n$ be the symmetric group of degree $n$.
We denote by $Y$ the direct sum of all non-projective Young modules
over the group algebra $k[S_n]$ of $S_n$. Then,
 for  every finite admissible subset $\Phi$ of $\mathbb N$, the
algebras $\scr G^{\Phi}_{k[S_n]}(k[S_n]\oplus Y)$ and
$\scr G^{\Phi}_{k[S_n]}(k[S_n]\oplus \Omega^i(Y))$ are stably equivalent
of Morita type for all $i\in \mathbb Z$.
\end{Koro}

\begin{Exam}

Let $A=k[x]/(x^n)$. Then $A$ is a representation-finite self-injective algebra.
Denote the indecomposable $A$-module by
$$X_r:=k[x]/(x^r)$$
 for $r=1,2,\cdots,n$.
Corollary \ref{7.1} thus shows that $\scr G^{\Phi}_{A}(A\oplus X_r)$ and $\scr G^{\Phi}_{A}(A\oplus X_{n-r})$ are stable equivalent of Morita type with $\Omega(X_r)=X_{n-r}$.

In the following, let $\Phi=\{0,1\}$ and $r=1$. 
Then we describe $\Phi$-Beilinson-Green algebras $\scr G^{\Phi}_{A}(A\oplus X_r)$ and $\scr G^{\Phi}_{A}(A\oplus X_{n-r})$ in terms of quivers with relations as follows
$$
\begin{array}{c}
 \xymatrix{  
 1 \ar@(dl,ul)[]^{\alpha} \ar@/^/[rr]| \beta & & 2 \ar@/^/[ll]|\gamma\ar[d]^{\delta}\\ 
  3 \ar@(dl,ul)[]^{\alpha'} \ar@/^/[rr]| {\beta'} & & 4 \ar@/^/[ll]|{\gamma'},  }\\
  \gamma\beta=\gamma\alpha=\alpha\beta=\beta\gamma-\alpha^{n-1}=\beta\delta=\delta\gamma'=0,\\
  \gamma'\beta'=\gamma'\alpha'=\alpha'\beta'=\beta'\gamma'-\alpha'^{n-1}=0.
  \end{array}
$$
and
$$
\begin{array}{c}

  \xymatrix{  
   1\ar@/^/[rr]| x & & 2 \ar@/^/[ll]| y\ar[d]^{\eta}\\ 
  3 \ar@/^/[rr]|{x'} & & 4 \ar@/^/[ll]| {y'}  } \\
 (xy)^{n-1}= (x'y')^{n-1}=x\eta=\eta y'=0.   
 \end{array}   
  $$

We calculate that $\dk(\scr G^{\Phi}_{A}(A\oplus X_1))=2n+7, \dk(\scr G^{\Phi}_{A}(A\oplus X_{n-1}))=4n-2$. 
Motivated by Zheng's method \cite[Property 5.1]{Zheng}, it is an interesting question to show that whether the  $\Phi$-Beilinson-Green algebras $\scr G^{\Phi}_{A}(A\oplus X_r)$ and $\scr G^{\Phi}_{A}(A\oplus X_{n-r})$ are cellular algebras. Note that cellular algebras were introduced by Graham and Lehrer \cite{GL} and the theory of cellular algebra related to the modular representations of Hecke algebras and various topics.
\end{Exam}

\bigskip

\end{document}